\documentclass[12pt]{article}
\usepackage{amsmath}
\usepackage{amsthm}
\usepackage{amsfonts}
\usepackage{amssymb}
\usepackage[cmtip,all]{xy}
\usepackage{shuffle}
\theoremstyle{plain}
\newtheorem{thm}{Theorem}
\newtheorem{cor}{Corollary}
\newtheorem{prop}{Proposition}

\theoremstyle{remark}
\newtheorem*{remark}{Remark}
\numberwithin{thm}{section}
\numberwithin{cor}{section}
\numberwithin{prop}{section}
\begin{document}
\title{Quasi-shuffle products revisited}
\author{Michael E. Hoffman and Kentaro Ihara\\
\small U. S. Naval Academy\\[-0.8 ex]
\small Annapolis, MD 21402 USA\\[-0.8 ex]
\small and\\[-0.8ex]
\small Kindai University\\[-0.8ex]
\small 3-4-1 Kowakae Higashi-Osaka 577-8502 Japan\\[-0.8ex]
\small \texttt{meh@usna.edu}\\[-0.8ex]
\small \texttt{k-ihara@math-kindai.ac.jp}}
\date{\small 22 March 2017
\small 2010 AMS Classifications:  16T30, 11M32}
\maketitle
\def\sh{\shuffle}
\def\la{\lambda}
\def\ep{\epsilon}
\def\zt{\zeta}
\def\ga{\gamma}
\def\om{\omega}
\def\si{\sigma}
\def\Ga{\Gamma}
\def\pit{\frac{\pi}2}
\def\pif{\frac{\pi}4}
\def\De{\Delta}
\def\rDe{\tilde\De}
\def\Si{\Sigma}
\def\op{\diamond}
\def\D{\mathcal D}
\def\N{\mathbb N}
\def\Z{\mathbb Z}
\def\Q{\mathbb Q}
\def\R{\mathbb R}
\def\C{\mathbb C}
\def\CC{\mathcal C}
\def\H{\mathfrak H}
\def\HH{\mathcal H}
\def\E{\mathfrak E}
\def\P{\mathcal P}
\def\<{\langle}
\def\>{\rangle}
\def\kA{k\langle A\rangle}
\def\id{\operatorname{id}}
\def\Hom{\operatorname{Hom}}
\def\L{\operatorname{Li}}
\def\QS{\operatorname{QSym}}
\def\Sy{\operatorname{Sym}}
\def\sgn{\operatorname{sgn}}
\def\sech{\operatorname{sech}}
\newcommand{\stone}[2]{\genfrac{[}{]}{0pt}{}{#1}{#2}}
\begin{abstract}
Quasi-shuffle products, introduced by the first author, have been
useful in studying multiple zeta values and some of their analogues
and generalizations.  The second author, together with Kajikawa, Ohno, 
and Okuda, significantly extended the definition of quasi-shuffle 
algebras so it could be applied to multiple $q$-zeta values.  
This article extends some of the algebraic machinery of the first 
author's original paper to the more general definition, and 
demonstrates how various algebraic formulas in the quasi-shuffle
algebra can be obtained in a transparent way.
Some applications to multiple zeta values, interpolated multiple
zeta values, multiple $q$-zeta values, and multiple polylogarithms 
are given.
\end{abstract}
\section{Introduction}
This article revisits the construction of quasi-shuffle products 
in \cite{H2}.
In \cite{IKOO} the construction of \cite{H2} was put in a more
general setting that had two chief advantages:  (i) it simultaneously
applied to multiple zeta and multiple zeta-star values and
their extensions; and (ii) it could be applied to the $q$-series
version of multiple zeta values studied in \cite{B}.
Here we show that some of the algebraic machinery developed in \cite{H2},
particularly the coalgebra structure and linear functions induced by
formal power series (not considered in \cite{IKOO}),
can be carried over to the more general setting and used to make
the calculations in the quasi-shuffle algebra, including many of \cite{IKOO}, 
more transparent.
We also describe some applications of quasi-shuffle algebras
not considered in \cite{IKOO}, including applications to the 
interpolated multiple zeta values introduced in \cite{Y1}
\par
The original quasi-shuffle product was inspired by the multiplication
of multiple zeta values, i.e.,
\begin{equation}
\label{mzv}
\sum_{n_1>\dots>n_k\ge 1}\frac1{n_1^{i_1}\cdots n_k^{i_k}} ,
\end{equation}
with $i_1>1$ to insure convergence.
One can associate to the series (\ref{mzv}) the monomial $z_{i_1}\cdots z_{i_k}$
in the noncommuting variables $z_1,z_2,\dots$; then we write the
value (\ref{mzv}) as $\zt(z_{i_1}\cdots z_{i_k})$.
For any monomials $w=z_iw'$ and $v=z_jv'$, define the product $w*v$
recursively by
\begin{equation}
\label{recur}
w*v=z_i(w'*v)+z_j(w*v')+z_{i+j}(w'*v') .
\end{equation}
Then $\zt(w)\zt(v)=\zt(w*v)$, where we think of $\zt$ as a linear
function on monomials.  As we shall see in the next section, the
recursive rule (\ref{recur}) is a quasi-shuffle product on 
monomials in the $z_i$ derived from the product $\op$ on the vector
space of $z_i$'s given by $z_i\op z_j=z_{i+j}$.
\par
In \cite{B} the multiple $q$-zeta values were defined as
\begin{equation}
\label{qmzv}
\sum_{n_1>\dots>n_k\ge 1}\frac{q^{(i_1-1)n_1}\cdots q^{(i_k-1)n_k}}
{[n_1]_q^{i_1}\cdots [n_k]_q^{i_k}} ,
\end{equation}
where $[n]_q=1+q+\dots+q^{n-1}=(1-q^n)/(1-q)$. 
If we denote (\ref{qmzv}) by $\zt_q(z_{i_1}\cdots z_{i_k})$, then
to have $\zt_q(w)\zt_q(v)=\zt_q(w*v)$
the recursion (\ref{recur}) must be significantly modified:
in place of $z_i\op z_j=z_{i+j}$ we must have
\[
z_i\op z_j=z_{i+j}+(1-q)z_{i+j-1} .
\]
This means that to have a theory of quasi-shuffle algebras that
applies to multiple $q$-zeta values, two restrictions in the
original construction of \cite{H2} must be removed:  that 
the product $a\op b$ of two letters be a letter, and that 
the operation $\op$ preserve a grading.  This was done in
\cite{IKOO}.  The same paper also addressed the relation between
multiple zeta values (\ref{mzv}) and the multiple zeta-star values
\begin{equation}
\label{mzvstar}
\zt^{\star}(z_{i_1}\cdots z_{i_k})=
\sum_{n_1\ge\dots\ge n_k\ge 1}\frac1{n_1^{i_1}\cdots n_k^{i_k}} .
\end{equation}
This relation can be expressed in terms of a linear isomorphism
(here denoted $\Si$) from the vector space of monomials in the
$z_i$'s to itself.
The function $\Si$ acts on monomials as, e.g.,
\[
\Si(z_iz_jz_k)=z_iz_jz_k+(z_i\op z_j)z_k+z_i(z_j\op z_k)+z_i\op z_j \op z_k
\]
and then $\zt^{\star}(w)=\zt(\Si(w))$.
If we define analogously multiple $q$-zeta-star values
$\zt_q^{\star}(w)$, then $\zt_q^{\star}(w)=\zt_q(\Si(w))$.
\par 
Important properties of $\Si$ were established in \cite{IKOO},
though some of the inductive proofs are tedious.  Here we
make use of two aspects of the theory developed in \cite{H2}
not used in \cite{IKOO}.  First, for any formal power series 
\[
f=c_1t + c_2 t^2+\cdots ,
\]
it is possible to define a linear function (but not necessarily an 
algebra homomorphism) $\Psi_f$ from (the vector space underlying) the 
quasi-shuffle algebra to itself.
This process respects composition (i.e., $\Psi_{f\circ g}=\Psi_f\Psi_g$), 
and many important isomorphisms can be represented this way,
e.g., $\Si=\Psi_{\frac{t}{1-t}}$.
Second, the quasi-shuffle algebra together with the ``deconcatenation''
coproduct is a Hopf algebra:  in fact, it turns out that its antipode
is closely related to $\Si$.
\par
This paper is organized as follows.
In \S2 we define the quasi-shuffle products $*$ and $\star$ on
the vector space $\kA$, where $A$ is a countable set of letters,
and $kA$ is equipped with a commutative product $\op$.
Then in \S3 we explain how to obtain linear isomorphisms from
$\kA$ to itself using formal power series:
as noted above, this gives a useful representation of $\Si$.
In \S4 we describe three Hopf algebras:  the ordinary Hopf
algebras $(\kA,*,\De)$ and $(\kA,\star,\De)$, and the infinitesimal
Hopf algebra $(\kA,\op,\rDe)$, where $\De$ is deconcatenation,
$\rDe(w)=\De(w)-w\otimes 1-1\otimes w$, and $\op$ is an extension
of the original operation on $A$ to a (noncommutative) product
on $\kA$.  Each of these Hopf algebras is associated with a
represention of $\Si$ via the antipode.  In \S5 we apply the
machinery of the preceding two sections to obtain many of
the algebraic formulas of \cite{IKOO} (and generalizations
thereof) in a transparent way.  Finally, in \S6 we illustrate
some of these algebraic formulas (particularly Corollary \ref{expthm}
below in each case) for five different homomorphic images of 
quasi-shuffle algebras.
\section{The quasi-shuffle products}
We start with a field $k$ containing $\Q$, and a countable set $A$
of ``letters''.  We let $kA$ be the vector space with $A$
as basis, and suppose there is an associative and commutative 
product $\op$ on $kA$.
\par
Now let $\kA$ be the noncommutative polynomial algebra over $A$.
So $\kA$ is the vector space over $k$ generated by ``words'' 
(monomials) $a_1a_2\cdots a_n$, with $a_i\in A$:  a word
$w=a_1\cdots a_n$ has length $\ell(w)=n$.
(We think of 1 as the empty word, and set $\ell(1)=0$.)
Following \cite{IKOO}, we define two $k$-bilinear products $*$ and
$\star$ on $\kA$ by making $1\in\kA$ the identity element for each 
product, and requiring that $*$ and $\star$ satisfy the relations
\begin{align}
\label{ast}
aw*bv &=a(w*bv)+b(aw*v)+(a\op b)(w*v)\\
\label{star}
aw\star bv &=a(w\star bv)+b(aw\star v)-(a\op b)(w\star v) 
\end{align}
for all $a,b\in A$ and all monomials $w,v$ in $\kA$.
As in \cite{H2} we have the following result.
\begin{thm}
\label{algebra}
If equipped with either the product $*$ or the product $\star$,
the vector space $\kA$ becomes a commutative algebra.
\end{thm}
\begin{proof}
We prove the result for $*$, as the proof for $\star$ is almost
identical.  It suffices to show that $*$ is commutative and associative.
For commutativity, it is enough to show that $u_1*u_2=u_2*u_1$ 
for words $u_1,u_2$:  we proceed by induction on $\ell(u_1)+\ell(u_2)$.  
This is trivial if either $u_1$ or $u_2$ is empty, so write 
$u_1=aw$ and $u_2=bv$ for $a,b\in A$ and words $w,v$.  
Then by Eq. (\ref{ast}),
\[
u_1*u_2-u_2*u_1=(a\op b)(w*v)-(b\op a)(v*w),
\]
and the right-hand side is zero by the induction hypothesis and
the commutativity of $\op$.  
\par
Similarly, to prove associativity it is enough to show that 
$u_1*(u_2*u_3)=(u_1*u_2)*u_3$ for words $u_1,u_2,u_3$, and this can 
be done by induction on $\ell(u_1)+\ell(u_2)+\ell(u_3)$.
The required identity is trivial if any of $u_1,u_2,u_3$ is 1,
so we can write $u_1=aw$, $u_2=bv$, and $u_3=cy$ for $a,b,c\in A$
and words $w,v,y$.  Then
\begin{multline*}
u_1*(u_2*u_3)-(u_1*u_2)*u_3=
a(w*b(v*cy)+b(aw*(v*cy))+(a\op b)(w*(v*cy))\\
+a(w*c(bv*y))+c(aw*(bv*y))+(a\op c)(w*(bv*y))\\
+a(w*(b\op c)(v*y))+(b\op c)(aw*(v*y))+(a\op(b\op c))(w*(v*y))\\
-a((w*bv)*cy)-c(a(w*bv)*y)-(a\op c)((w*bv)*y)\\
-b((aw*v)*cy)-c(b(aw*v)*y)-(b\op c)((aw*v)*y)\\
-(a\op b)((w*v)*cy)-c((a\op b)(w*v)*y)-((a\op b)\op c)((w*v)*y)=\\
a(w*(bv*cy))+c(aw*(bv*y))-a((w*bv)*cy)-c((aw*bv)*y)=0,
\end{multline*}
by the induction hypothesis and the associativity of $\op$.
\end{proof}
If the product $\op$ is identically zero, then $*$ and $\star$
coincide with the usual shuffle product $\sh$ on $\kA$.  We call
both $*$ and $\star$ quasi-shuffle products.
\par
We note that $\op$ can be extended to a product of on all of
$\kA$ by defining $1\op w=w\op 1=w$ for all words $w$, and
$w\op v=w'(a\op b)v'$ for nonempty words $w=w'a$ and $v=bv'$
(where $a,b$ are letters).  Then $(\kA,\op)$ is a noncommutative
algebra that contains the commutative subalgebra $k1+kA$.
\section{Linear maps induced by power series}
In this section we show how any formal power series $f\in tk[[t]]$
induces a linear function $\Psi_f$ from $\kA$ to itself, and
also how various isomorphisms among the algebras $(\kA,\sh)$,
$(\kA,*)$ and $(\kA,\star)$ can be recognized as being of the
form $\Psi_f$.  This has been treated in an operadic context by
Yamamoto \cite{Y2}.
\par
Let $a_1,a_2,\dots a_n\in A$.  If $w=a_1a_2\cdots a_n$, and
$I=(i_1,\dots,i_m)$ is a composition of $n$ (i.e., a sequence
of positive integers whose sum is $n$), define (as in \cite{H2})
\begin{equation}
\label{Idef}
I[w]=(a_1\op\cdots\op a_{i_1})(a_{i_1+1}\op\cdots\op
a_{i_1+i_2})\cdots(a_{i_1+\dots+i_{m-1}+1}\op\cdots\op a_n) .
\end{equation}
We call $n=|I|$ the weight of the composition of $I$, and $m=\ell(I)$
its length.
Note that the parentheses in Eq. (\ref{Idef}) are not really
necessary:  the right-hand side is simultaneously an $m$-fold
product in the concatenation algebra $\kA$ and a product of length
\[
1+(i_1-1)+(i_2-1)+\dots+(i_m-1)=n+1-m
\]
in the algebra $(\kA,\op)$.  If we define
\[
I\<w\>=a_1\cdots a_{i_1}\op a_{i_1+1}\cdots a_{i_1+i_2}\op \cdots \op
a_{i_1+\dots+i_{m-1}+1}\cdots a_n
\]
so that, e.g.,
\[
(2,1,2)[a_1a_2a_3a_4a_5]=a_1\op a_2a_3a_4\op a_5=(1,3,1)\<a_1a_2a_3a_4a_5\> ,
\]
then $I[w]=I^*\<w\>$ defines an involution $*$ on compositions such that
$|I^*|=|I|$ and $\ell(I^*)=|I|+1-\ell(I)$.
\par
Now we consider formal power series
\begin{equation}
\label{eff}
f=c_1t+c_2t^2+c_3t^3+\cdots\in tk[[t]].
\end{equation}
Any two ``functions'' $f,g\in tk[[t]]$, say
$f=\sum_{i\ge 1} c_i t^i$ and $g=\sum_{i\ge 1} d_i t^i$,
have a ``functional composition''
\begin{multline*}
f\circ g=\sum_{i\ge 1}c_ig^i =
c_1(d_1t+d_2t^2+\cdots)+c_2(d_1t+d_2t^2+\cdots)^2+\cdots\\
=c_1d_1t + (c_1d_2+c_2d_1^2)t^2+\cdots\in tk[[t]] .
\end{multline*}
Writing $[t^i]f$ for the coefficient of $t^i$ in $f\in k[[t]]$,
it is evident that
\begin{equation}
\label{comp}
[t^k]f\circ g=\sum_{j=1}^k [t^j]f[t^k]g^j .
\end{equation}
Clearly $f=t$ is the identity for functional composition; and if 
$\P\subset tk[[t]]$ is the set of power series invertible under
functional composition, it is not hard to see that $f\in\P$ if 
and only if $[t]f\ne 0$.
\par
For $f$ given by Eq. (\ref{eff}), we define the $k$-linear map
$\Psi_f:\kA\to\kA$ by $\Psi_f(1)=1$ and
\begin{equation}
\label{phif}
\Psi_f(w)=\sum_{I=(i_1,\dots,i_m)\in\CC(\ell(w))}c_{i_1}\cdots c_{i_m}I[w] ,
\end{equation}
for nonempty words $w$, where $\CC(n)$ is the set of compositions of $n$.  
The following result generalizes Lemma 2.4 of \cite{H2}.
\begin{thm}
\label{funct}
For $f,g\in tk[[t]]$, $\Psi_f\Psi_g=\Psi_{f\circ g}$.
\end{thm}
\begin{proof}  Since
\[
\Psi_g(w)=\sum_{I=(i_1,\dots,i_m)\in\CC(\ell(w))}[t^{i_1}]g\cdots
[t^{i_m}]g I[w]
\]
we have
\begin{multline*}
\Psi_f\Psi_g(w)=\\
\sum_{m=1}^{\ell(w)}
\sum_{J=(j_1,\dots,j_l)\in\CC(m)}\sum_{I=(i_1,\dots,i_m)\in\CC(\ell(w))}
[t^{j_1}]f\cdots [t^{j_l}]f[t^{i_1}]g\cdots [t^{i_m}]g J[I[w]] .
\end{multline*}
\par
On the other hand,
\[
\Psi_{f\circ g}(w)=\sum_{K=(k_1,\dots,k_l)\in\CC(\ell(w))}[t^{k_1}]f\circ g\cdots
[t^{k_l}]f\circ g K[w] ,
\]
so we need to show that, for all compositions $K=(k_1,\dots,k_l)\in\CC(n)$,
\begin{multline}
\label{equ}
[t^{k_1}]f\circ g \cdots [t^{k_l}]f\circ g=\\
\sum_{m=1}^n \sum_{J=(j_1,\dots,j_l)\in\CC(m)}
\sum_{\substack{I=(i_1,\dots,i_m)\in\CC(n) \\ JI=K}}
[t^{j_1}]f\cdots [t^{j_l}]f[t^{i_1}]g\cdots [t^{i_m}]g 
\end{multline}
where 
\[
JI=(i_1+\dots+i_{j_1},i_{j_1+1}+\dots+i_{j_1+j_2},\dots,
i_{j_1+\dots+j_{l-1}+1}+\dots+i_m)
\] 
is the obvious ``composition'' of the compositions 
$I=(i_1,\dots,i_m)$ and $J=(j_1,\dots,j_l)$, with $J\in\CC(m)$.
Now the right-hand side of Eq. (\ref{equ}) can be rewritten
\begin{multline*}
\sum_{J=(j_1,\dots,j_l)}\sum_{\substack{I=(i_1,\dots,i_{|J|})\in\CC(n) \\ JI=K}}
\prod_{s=1}^l [t^{j_s}]f[t^{i_{j_1+\dots+j_{s-1}+1}}]g\cdots
[t^{i_{j_1+\dots+j_s}}]g \\
=\sum_{J=(j_1,\dots,j_l)}\prod_{s=1}^l [t^{j_s}]f[t^{k_s}]g^{j_s} ,
\end{multline*}
from which Eq. (\ref{equ}) follows by use of (\ref{comp}).
\end{proof}
\subsection{The isomorphisms $T$ and $\Si$}
From the preceding result, $\Psi_f$ is an isomorphism when $f\in\P$.
We now consider some particular examples.
First, it is immediate from Eq. (\ref{phif}) that $\Psi_t$ is the identity 
homomorphism of $\kA$.  
Now, following \cite{IKOO}, consider
\[
T=\Psi_{-t}\quad\text{and}\quad \Si=\Psi_{\frac{t}{1-t}} .
\]
(The function we call $\Si$ is written $S$ in \cite{IKOO}, but
as in \cite{H2} we wish to reserve $S$ for a Hopf algebra
antipode.)
For words $w$ of $\kA$, $T(w)=(-1)^{\ell(w)}w$ and
\[
\Si(w)=\sum_{I\in\CC(\ell(w))}I[w] .
\]
Evidently $T$ is an involution, and $\Si^{-1}=\Psi_{\frac{t}{1+t}}$
is given by
\[
\Si^{-1}(w)=\sum_{I\in\CC(\ell(w))}(-1)^{\ell(w)-\ell(I)}I[w] .
\]
We note that, for letters $a$ and words $w\ne 1$,
\begin{align}
\label{Tind}
T(aw)&=-aT(w)\\
\label{Sind}
\Si(aw)&= a\Si(w) + a\op \Si(w)\\
\label{Siind}
\Si^{-1}(aw)&=a\Si^{-1}(w)-a\op \Si^{-1}(w)
\end{align}
and (as in \cite{IKOO}) the property (\ref{Sind}) can be used to define $\Si$.
The functions $\Si$ and $T$ are not inverses, but we have the following result.
\begin{cor} The functions $\Si$ and $T$ satisfy $T\Si T=\Si^{-1}$, and
(if the product $\op$ is nonzero) generate the infinite dihedral group.
\end{cor}
\begin{proof} From Theorem \ref{funct} we have $\Si^n=\Psi_{\frac{t}{1-nt}}$,
so all powers of $\Si$ are distinct (unless $\op=0$, in which case
$\Si=\id$).
We have also $T\Si T=\Psi_{\frac{t}{1+t}}
=\Si^{-1}$.
\end{proof}
\par\noindent
It follows immediately that $T\Si$ and $\Si T$ are involutions (cf. 
\cite[Prop. 2]{IKOO}).
For future reference we note that the equation $\Si^p=\Psi_{\frac{t}{1-pt}}$
defines $\Si^p$ for any $p\in k$:  from Theorem \ref{funct} we have
$\Si^p\Si^q=\Si^{p+q}$, and $\Si^p$ is the $p$th iterate of $\Si$ when
$p$ is an integer.
\par
From \cite{H2} we have the (inverse) functions $\exp=\Psi_{e^t-1}$ and
$\log=\Psi_{\log(1+t)}$.  As shown in \cite[Theorem 2.5]{H2},
$\exp$ is an algebra isomorphism from $(\kA,\sh)$ to $(\kA,*)$.
The functions $\exp$ and $\log$ are related to $\Si$ and $T$ as
follows.
\begin{cor}
\label{siform}
$\Si=\exp T\log T$ .
\end{cor}
\begin{proof} This is immediate from Theorem \ref{funct}, 
since $\exp T=\Psi_{e^{-t}-1}$,
$\log T=\Psi_{\log(1-t)}$, and $\log(1-t)$ composed with $e^{-t}-1$ gives
\[
\frac1{1-t}-1=\frac{1-(1-t)}{1-t}=\frac{t}{1-t} .
\]
\end{proof}
\par
We now turn to the algebraic properties of $T$ and $\Si$.
\begin{prop}
\label{Thom}
$T:(\kA,*)\to (\kA,\star)$ and $T:(\kA,\star)\to (\kA,*)$
are homomorphisms.
\end{prop}
\begin{proof}
We prove the first statement; the second then follows because $T$ is
an involution.  We shall show that $T(u_1*u_2)=T(u_1)\star T(u_2)$
for any words $u_1,u_2$ by induction on $\ell(u_1)+\ell(u_2)$.  
The result is immediate if $u_1$ or $u_2$ is 1, so write $u_1=aw$ and 
$u_2=bv$ for letters $a,b$ and words $w,v$.  Then
\begin{align*}
T(u_1*u_2)&=T(a(w*bv)+b(aw*v)+(a\op b)(w*v))\\
&=-a(T(w)*T(bv))-b(T(aw)*T(v))-(a\op b)(T(w)*T(v))\\
&=a(T(w)*bT(v))+b(aT(w)*T(v))-(a\op b)(T(w)*T(v))\\
&=aT(w)\star bT(v)=T(u_1)\star T(u_2) ,
\end{align*}
where we have used the induction hypothesis and Eq. (\ref{Tind}).
\end{proof}
The following result was proved as Theorem 1 in \cite{IKOO} in a much 
less direct way.
\begin{cor}
\label{Shom}
The linear isomorphism $\Si:(\kA,\star)\to (\kA,*)$ is an algebra
isomorphism.
\end{cor}
\begin{proof}
This follows from Corollary \ref{siform}, since $\Si$ is the composition
\[
(\kA,\star)\xrightarrow{T} (\kA,*)\xrightarrow{\log} (\kA,\sh)
\xrightarrow{T}(\kA,\sh)\xrightarrow{\exp}(\kA,*)
\]
of homomorphisms (that $T$ is an endomorphism of $(\kA,\sh)$
follows by taking $\op$ to be the zero product in Proposition \ref{Thom}).
\end{proof}
In fact, the following is a commutative diagram of algebra isomorphisms:
\begin{equation}
\label{triangle}
\begin{split}
\xymatrixcolsep{3pc}
\xymatrix{
& (\kA,*) \\
(\kA,\sh) \ar[ru]^{\exp} \\
& (\kA,\star) \ar[lu]^{T\log T} \ar[uu]^{\Si}
}
\end{split}
\end{equation}
\begin{cor}
\label{STaut}
The involutions $\Si T:(\kA,*)\to (\kA,*)$ and 
$T\Si:(\kA,\star)\to (\kA,\star)$ are algebra automorphisms.
\end{cor}
\begin{proof}
Immediate from Proposition \ref{Thom} and Corollary \ref{Shom}.
\end{proof}
\subsection{A one-parameter family of automorphisms}
Let $p\ne 0$ be an element of $k$, and set 
\[
H_p=\exp \Psi_{pt} \log
\]
Evidently $H_1=\id$ and $H_pH_q=H_{pq}$, so this is a one-parameter
family of isomorphisms of the vector space $\kA$.  
We can write $H_p=\Psi_{(1+t)^p-1}$, where $(1+t)^p-1$ is the power series
\[
\sum_{n\ge 1}\binom{p}{n}t^n=pt +\frac{p(p-1)}{2!}t^2+\frac{p(p-1)(p-2)}{3!}t^3
+\cdots .
\]
\par
From Corollary \ref{siform} $H_{-1}=\Si T$, so
\[
H_{-1}(w*v)=\Si T(w*v)=\Si(T(w)\star T(v))=H_{-1}(w)*H_{-1}(v)
\]
for any words $w,v$.  In fact, this property holds for all $p$.
\begin{thm}
\label{hpthm}
For all $p\ne 0$ and words $w,v$, $H_p(w*v)=H_p(w)*H_p(v)$.
\end{thm}
\begin{proof}  Since
\[
\Psi_{pt}(w)=p^{\ell(w)}w
\]
for all words $w$, it follows that
\[
\Psi_{pt}(w\sh v)=\Psi_{pt}(w)\sh\Psi_{pt}(v)
\]
for all words $w,v$. 
Hence, since $\log(w*v)=\log w\sh\log v$,
\begin{multline*}
H_p(w*v)=\exp(\Psi_{pt}(\log w\sh\log v))=
\exp(\Psi_{pt}(\log w)\sh\Psi_{pt}(\log v))\\
=H_p(w)*H_p(v).
\end{multline*}
\end{proof}
\par
Thus, $H_p$ is an automorphism of the algebra $(\kA,*)$.
Note also that $TH_pT=\Psi_{1-(1-t)^p}$ is an automorphism of
the algebra $(\kA,\star)$.
\section{Hopf algebra structures}
As in \cite{H2} we define a coproduct $\De$ on $\kA$ by
\[
\De(w)=\sum_{uv=w} u\otimes v ,
\]
for words $w$, where the sum is over all pairs $(u,v)$ of words
with $uv=w$ including $(1,w)$ and $(w,1)$, and a counit $\ep:\kA\to k$
by $\ep(1)=1$ and $\ep(w)=0$ for $\ell(w)>0$.
It will also be convenient to define the reduced coproduct $\rDe$
by $\rDe(1)=0$ and $\rDe(w)=\De(w)-w\otimes 1-1\otimes w$ for 
nonempty words $w$.
\par
The coproduct can be used to define a convolution product on
the set $\Hom_k(\kA,\kA)$ of $k$-linear maps from $\kA$ to itself,
which we denote by $\odot$:  for $L_1,L_2\in\Hom_k(\kA,\kA)$ and
words $w$ of $\kA$,
\[
L_1\odot L_2(w)=\sum_{uv=w} L_1(u)L_2(v) .
\]
(The reader is warned that this is {\it not} the usual convolution
for either of the Hopf algebras $(k\<A\>,*,\De)$ or $(k\<A\>,\star,\De)$ 
defined below.)
The convolution product $\odot$ has unit element $\eta\ep$, 
where $\eta:k\to\kA$ is the unit map (i.e., it sends $1\in k$ to $1\in\kA$).
It is easy to show that any $L\in\Hom_k(\kA,\kA)$ with $L(1)=1$ has
a convolutional inverse, which we denote by $L^{\odot(-1)}$.
\par
We call $C\in\Hom_k(\kA,\kA)$ a contraction if $C(1)=0$ and $C(w)$ is
primitive (i.e., $\rDe C(w)=0$) for all words $w$, and $E\in\Hom_k(\kA,\kA)$ 
an expansion if $E(1)=1$ and $E$ is a coalgebra map.  If $C$ is a
contraction and $E$ is an expansion, we say $(E,C)$ is an inverse
pair if
\begin{equation}
\label{ctoe}
E=(\eta\ep - C)^{\odot (-1)}=\eta\ep+C+C\odot C+C\odot C\odot C+\cdots
\end{equation}
or equivalently
\begin{equation}
\label{etoc}
C=\eta\ep - E^{\odot (-1)}
\end{equation}
\begin{prop}
\label{expcon}
Suppose $C\in\Hom_k(\kA,\kA)$ is a contraction and $E$ is given by 
Eq. (\ref{ctoe}).  Then $(E,C)$ is an inverse pair.  
Conversely, if $E\in\Hom_k(\kA,\kA)$ is an expansion and $C$
is given by Eq. (\ref{etoc}), then $(E,C)$ is an inverse pair.
\end{prop}
\begin{proof}
Suppose first that $C$ is a contraction.  Evidently $E(1)=1$
from Eq. (\ref{ctoe}), so it suffices to show $E$ a coalgebra map.
Now Eq. (\ref{ctoe}) implies
\[
E(w)=\sum_{u_1\cdots u_n=w} C(u_1)\cdots C(u_n) 
\]
for words $w\ne 1$, 
where the sum is over all decompositions $w=u_1\cdots u_n$ into subwords
$u_i\ne 1$.  Hence 
\begin{multline*}
\De E(w)=E(w)\otimes 1 +1\otimes E(w)+\\
\sum_{u_1\cdots u_n=w, n\ge 2}
\sum_{i=1}^{n-1} C(u_1)\cdots C(u_i)\otimes C(u_{i+1})\cdots C(u_n) ,
\end{multline*}
which can be seen to agree with $(E\otimes E)\De(w)$.
\par
Now suppose $E$ is an expansion.  Eq. (\ref{etoc}) implies
$C(1)=0$, so it suffices to show $C(w)$ primitive for words
$w$.  We proceed by induction on $\ell(w)$.  Suppose $C$
primitive on all words of length $<n$, and let $\ell(w)=n$.
Then Eq. (\ref{etoc}) implies
\[
C(w)=E(w)-\sum_{uv=w, v\ne 1}C(u)E(v) ,
\]
and by the induction hypothesis it follows that $\De C(w)$ can be written
\begin{multline*}
\De E(w)-\sum_{uv=w, v\ne 1}(C(u)\otimes 1)\De E(v)
-\sum_{uv=w, v\ne 1}1\otimes C(u)E(v) =\\
C(w)\otimes 1+\sum_{\substack{uv=w \\ u\ne 1\ne v}}E(u)\otimes E(v)-
\sum_{\substack{uv_1v_2=w \\ v_2\ne 1}}C(u)E(v_1)\otimes E(v_2)+1\otimes C(w) .
\end{multline*}
Then
\[
\rDe C(w)=
\sum_{uv=w, u\ne 1\ne v}\left[E(u)-\sum_{u_1u_2=u} C(u_1)E(u_2)\right]\otimes E(v),
\]
and the quantity in brackets is zero by Eq. (\ref{etoc}).
\end{proof}
\par
Now let $f=c_1t+c_2t^2+\cdots\in tk[[t]]$, and let $\Psi_f$ be 
the corresponding linear map of $\kA$ as defined in \S3.  
Define the linear map $C_f:\kA\to kA$ by $C_f(1)=0$ and 
$C_f(a_1a_2\cdots a_n)=c_n a_1\op a_2\op\cdots\op a_n$ for 
$a_1, a_2,\dots a_n\in A$.
Then we have the following result.
\begin{thm}
\label{coalgid} 
For any $f\in tk[[t]]$, $(\Psi_f,C_f)$ is an inverse pair.
\end{thm}
\begin{proof}
It is evident from definitions that $\Psi_f(1)=1$ and 
\[
\Psi_f(w)=\sum_{k=1}^n C_f(a_1\cdots a_k)\Psi_f(a_{k+1}\cdots a_n) .
\]
for $w=a_1\cdots a_n$ with $a_i\in A$ and $n\ge 1$.
Stated in terms of the convolution product, this is
\[
\Psi_f=C_f \odot \Psi_f + \eta\ep ,
\]
from which Eq. (\ref{ctoe}) (with $E=\Psi_f$, $C=C_f$) follows.
Since evidently $C_f$ is a contraction, the result follows.
\end{proof}
\par
For a word $w$ of $\kA$, say $w=a_1\cdots a_n$ with the $a_i\in A$,
the ``reverse'' of $w$ is $R(w)=a_na_{n-1}\cdots a_1$.  If we set
$R(1)=1$, then $R$ extends to a linear map from $\kA$ to itself,
which is evidently an involution.  
While $R$ is not a coalgebra map for $\De$ (despite the
incorrect statement in \cite{H3}), the following is true.
\begin{prop} 
$R$ is an automorphism of both $(\kA,*)$ and $(\kA,\star)$.
\end{prop}
\begin{proof}
In view of the next result and the commutative diagram (\ref{triangle}),
it suffices to prove that $R$ is an automorphism of $(\kA,\sh)$.
But this is evident from the well-known description of the shuffle 
product as 
\[
a_1a_2\cdots a_k\sh a_{k+1}a_{k+2}\cdots a_{k+l}=\sum_{\si\in S_{k,l}}
a_{\si(1)}a_{\si(2)}\cdots a_{\si(k+l)}
\]
where $a_1,\dots, a_{k+l}$ are letters and $S_{k,l}$ is the subgroup of the 
symmetric group $S_{k+l}$ consisting of all permutations $\si$ with 
$\si(1)<\si(2)<\dots<\si(k)$
and $\si(k+1)<\si(k+2)<\dots<\si(k+l)$.
\end{proof}
\begin{prop}
$R$ commutes with $\Psi_f$ for all $f\in tk[[t]]$.
\end{prop}
\begin{proof}
For a composition $I=(i_1,\dots,i_l)$, let $\bar I=(i_l,\dots,i_1)$.
Then evidently $R(I[w])=\bar I[R(w)]$.  Since $I\mapsto\bar I$ is
an involution of $\CC(n)$, the conclusion then follows from the 
definition (\ref{phif}) of $\Psi_f$.
\end{proof}
\begin{thm}
$(\kA,*,\De)$ and $(\kA,\star,\De)$ are Hopf algebras, with
respective antipodes $S_*=\Si TR$ and $S_{\star}=T\Si R$.
\end{thm}
\begin{proof}
The inductive argument in \cite[Theorem 3.1]{H2} that $\De$ is a 
homomorphism for $*$ works equally well for $\star$, so $(\kA,*,\De)$
and $(\kA,\star,\De)$ are bialgebras.
Although these bialgebras are not necessarily graded, they are 
filtered by word length:  $\kA^n$ is the subspace generated by words 
of length at most $n$.  
Since $\kA^0=k1$, these bialgebras are filtered connected, 
and thus automatically Hopf algebras (see, e.g., \cite{M}).  
In fact, the proof of the explicit formula for $S_*$
in \cite[Theorem 3.2]{H2} (by induction on word length) carries
over to this setting, giving
\begin{align*}
S_*(w)=(-1)^n \sum_{I\in\CC(n)}I[a_n a_{n-1}\cdots a_1] 
\end{align*}
for a word $w=a_1a_2\cdots a_n$ in $\kA$, i.e., $S_*(w)=\Si TR(w)$.
The antipode $S_{\star}$ of the Hopf algebra $(\kA,\star,\De)$ is uniquely 
determined by the conditions 
\begin{equation}
\label{starcon}
S_{\star}(1)=1 \quad\text{and}\quad
\sum_{uv=w} S_{\star}(u)\star v = 0\quad\text{for words $w\ne 1$} .
\end{equation}
Now $S_*$ satisfies
\[
\sum_{uv=w} S_*(Tu)*Tv = 0
\]
for $w\ne 1$; apply $T$ both sides to get
\[
\sum_{uv=w} TS_*T(u)\star v = 0,
\]
from which we see that $S_{\star}=TS_*T$ satisfies (\ref{starcon}).
Since $T$ commutes with $R$, this means that $S_{\star}=T\Si R$.
\end{proof}
Since $S_*$ and $S_{\star}$ are antipodes of commutative Hopf algebras,
they are involutions and algebra automorphisms of $(\kA,*)$ and
$(\kA,\star)$ respectively.  
Since $R$ commutes with $\Si$ and $T$, this gives another proof of
Corollary \ref{STaut}.
Note also that $S_*S_{\star}=\Si^2$ and $S_{\star}S_*=\Si^{-2}$.
\par
For any $f\in k[[t]]$, $\Psi_f$ is a coalgebra map by Theorem \ref{coalgid}.
In particular, the maps $H_p$ of the last section are automorphisms
of the Hopf algebra $(\kA,*,\De)$, and (\ref{triangle}) is a 
commutative diagram of Hopf algebra isomorphisms.
\par
Recall that $(\kA,\op)$ is a noncommutative algebra.
We will now show that $(\kA,\op,\rDe)$ is an infinitesimal Hopf algebra 
(see \cite{A} for definitions).
\begin{thm} 
\label{iha}
$(\kA, \op,\tilde\De)$ is an infinitesimal Hopf algebra,
with antipode $S_{\op}=-\Si^{-1}$.
\end{thm}
\begin{proof}
First we show that $(\kA,\op,\rDe)$ is an infinitesimal bialgebra,
i.e., that
\begin{equation}
\label{infdi}
\rDe(w\op v)=\sum_v (w\op v_{(1)})\otimes v_{(2)} + \sum_w w_{(1)}\otimes 
(w_{(2)}\op v) ,
\end{equation}
for words $w,v$, where 
\[
\rDe(w)=\sum_w w_{(1)}\otimes w_{(2)}\quad\text{and}\quad
\rDe(v)=\sum_v v_{(1)}\otimes v_{(2)} .
\]
Eq. (\ref{infdi}) is immediate if $w$ or $v$ is 1, so we can assume
both are nonempty.  Write $w=a_1\cdots a_n$ and $v=b_1\cdots b_m$, where
the $a_i$ and $b_i$ are letters.  If $n=1$ we have $\rDe(w)=0$, so
Eq. (\ref{infdi}) becomes
\[
\rDe(a_1\op v)=\sum_v (a_1\op v_{(1)})\otimes v_{(2)} ,
\]
which is evidently true.  The case $m=1$ is similar, so we can
assume that $n,m\ge 2$.  Then
\begin{multline}
\rDe(w\op v)=\sum_{j=1}^{n-1}a_1\cdots a_j\otimes a_{j+1}\cdots 
a_{n-1}a_n\op b_1b_2\cdots b_m +\\
\sum_{i=1}^{m-1} a_1\cdots a_{n-1}a_n\op b_1b_2\cdots b_i
\otimes b_{i+1}\cdots b_m ,
\end{multline}
and the right-hand side evidently agrees with that of Eq. (\ref{infdi}).
\par
To show $(\kA,\op,\rDe)$ an infinitesimal Hopf algebra we now need to 
show that it has an antipode, i.e., a function $S_{\op}\in\Hom_k(\kA,\kA)$ with
\[
\sum_{w} S_{\op}(w_{(1)})\op w_{(2)} + S_{\op}(w) + w = 0 = 
\sum_{w} w_{(1)}\op S_{\op}(w_{(2)})+w+S_{\op}(w)
\]
for any $w\in\kA$, where $\rDe(w)=\sum_w w_{(1)}\otimes w_{(2)}$.  
This follows from \cite[Prop. 4.5]{A}, but we shall prove that 
$S_{\op}=-\Si^{-1}$ 
by showing that $-\Si^{-1}$ satisfies the defining property.   
We prove that the equation 
\begin{equation}
\label{anti}
\sum_w \Si^{-1}(w_{(1)})\op w_{(2)} = -\Si^{-1}(w) + w 
\end{equation}
holds for all words $w$ by induction on the length of $w$.  
Evidently Eq. (\ref{anti}) is true if $\ell(w)\le 1$.  Now suppose
Eq. (\ref{anti}) holds for $w\ne 1$:  we prove it for $aw$, $a\in A$.
Since $\rDe(aw)=(a\otimes 1)\rDe(w)+a\otimes w$, we must show that
\[
\sum_w \Si^{-1}(aw_{(1)})\op w_{(2)} + \Si^{-1}(a)\op w = -\Si^{-1}(aw) + aw
\]
Using Eq. (\ref{Siind}), this is 
\begin{multline*}
a \sum_w \Si^{-1}(w_{(1)})\op w_{(2)} - a \op \sum_w \Si^{-1}(w_{(1)})\op w_{(2)} 
+a\op w\\
=-a\Si^{-1}(w)+a\op\Si^{-1}(w)+aw .
\end{multline*}
The conclusion then follows by use of the induction hypothesis (\ref{anti}).
The proof that
\[
\sum_w w_{(1)}\op \Si^{-1}(w_{(2)}) = w - \Si^{-1}(w) 
\]
is similar, except that in place of Eq. (\ref{Siind}) one needs 
the identity 
\[
\Si^{-1}(wa)=\Si^{-1}(w)a-\Si^{-1}(w)\op a
\]
for words $w$ and letters $a$.
\end{proof}
The algebra $(\kA,\op)$ has the canonical derivation $D=\op\rDe$,
i.e. $D(w)=0$ for words $w$ with $\ell(w)\le 1$ and
\[
D(a_1a_2\cdots a_n)=\sum_{i=1}^{n-1} a_1\cdots a_i\op a_{i+1}\cdots a_n
\]
for letters $a_1,\dots, a_n$, $n\ge 2$.  
We note that $D^n(w)=0$ whenever $n\ge\ell(w)$, so that
\[
e^{D}=\sum_{n=0}^\infty \frac{D^n}{n!} = \id + D + \frac{D^2}{2!}+\cdots
\]
makes sense as an element of $\Hom_k(\kA,\kA)$, and similarly
for $e^{-D}$.
By \cite[Prop. 4.5]{A}, $\Si^{-1}=-S_{\op}=e^{-D}$.  In fact, 
this can be sharpened as follows.
\begin{cor}
\label{dform}
For any $r\in k$, $\Si^r=e^{rD}$.
\end{cor}
\begin{proof}
By definition
\begin{equation}
\label{Sir}
\Si^r(w)=\Psi_{\frac{t}{1-rt}}(w)=\sum_{|I|=\ell(w)} r^{|I|-\ell(I)}I[w]
\end{equation}
for any word $w$ of $\kA$.  On the other hand, by \cite[Prop. 4.4]{A}
\[
\frac{D^k}{k!}=\op^{(k)}\tilde\De^{(k)} ,
\]
where $\op^{(k)}:\kA^{\otimes(k+1)}\to \kA$ and 
$\tilde\De^{(k)}:\kA\to\kA^{\otimes(k+1)}$
are respectively the iterated $\op$-product and coproduct maps.
Now for a word $w$ of $\kA$,
\[
\op^{(k)}\tilde\De^{(k)}(w)=\sum_{\ell(I)=k+1,|I|=\ell(w)} I\<w\>
=\sum_{\ell(I)=k+1,|I|=\ell(w)} I^*[w] 
\]
(where we recall the definition of $I\<w\>$ and $I^*$ from \S3),
and so
\begin{multline*}
e^{rD}(w)=\sum_{k\ge 0}r^k \op^k\tilde\De^{(k)}(w)
=\sum_{k\ge 0} r^k\sum_{\ell(I)=k+1, |I|=\ell(w)} I^*[w] \\
=\sum_{|I|=\ell(w)}r^{\ell(I)-1}I^*[w]=\sum_{|I|=\ell(w)}r^{\ell(I^*)-1}I[w]=
\sum_{|I|=\ell(w)}r^{|I|-\ell(I)}I[w],
\end{multline*}
which agrees with the right-hand side of Eq. (\ref{Sir}).
\end{proof}
\begin{cor}
\label{dhom}
For any $r\in k$, $\Si^r$ is an automorphism of $(\kA,\op)$.
\end{cor}
\begin{proof}
The exponential of a derivation is an automorphism 
\cite[sect. I.2]{J}, so this follows from the preceding result.
\end{proof}
\section{Exponentials and logarithms}
Let $\la$ be a formal parameter, $\bullet$ any of the symbols
$*$, $\star$, $\sh$, or $\op$, and let $f\in tk[[t]]$ be given by
Eq. (\ref{eff}).
Set
\[
f_{\bullet}(\la w)=\sum_{i\ge 1} \la^i c_i w^{\bullet i}\in\kA[[\la]] ,
\]
for $w\in\kA$, and use this to define a map $f_{\bullet}$
from $\la\kA[[\la]]$ to itself.
We write $\exp_{\bullet}(u)$ for $1+f_{\bullet}(u)$, where
$f=e^t-1$; and $\log_{\bullet}(1+u)$ for $f_{\bullet}(u)$, where
$f=\log(1+t)$.  Then for any $w\in\kA$, 
\[
\log_{\bullet}(\exp_{\bullet}(\la w))=\la w\quad\text{and}\quad
\exp_{\bullet}(\log_{\bullet}(1+\la w))=1+\la w ;
\]
and for $w,v\in\kA$ for $\bullet=*$ or $\bullet=\star$, and 
$w,v\in kA$ for $\bullet=\op$,
\begin{equation}
\label{expsum}
\exp_{\bullet}(\la(w+v))=\exp_{\bullet}(\la w)\bullet\exp_{\bullet}(\la v) .
\end{equation}
We extend the automorphisms $\Psi_f$ of $\kA$ to $\kA[[\la]]$
by setting $\Psi_f(\la)=\la$.
\par
The following result generalizes Lemma 3 of \cite{IKOO} .
\begin{thm}
\label{Ihafid}
For any $f=c_1t+c_2t^2+\cdots\in tk[[t]]$ and $z\in kA[[\la]]$,
\begin{equation*}
\Psi_f\left(\frac1{1-\la z}\right)=\frac1{1-f_{\op}(\la z)} .
\end{equation*}
\end{thm}
\begin{proof}
In fact, we shall show that
\begin{equation}
\label{ecid}
E\left(\frac1{1-\la z}\right)=\frac1{1-C(\la z+\la^2z^2+\cdots)}
\end{equation}
for any inverse pair $(E,C)$:  the conclusion then follows by
Theorem \ref{coalgid}, noting that $f_{\op}(\la z)=
C_f(\la z+\la^2z^2+\cdots)$.
We can write the left-hand side of Eq. (\ref{ecid}) as
\[
\left(\eta\ep+C+C\odot C+\cdots\right)(1+\la z+\la^2z^2+\cdots)
=1+\sum_{n\ge 1}\sum_{k\le n} C^{\odot k}(\la^nz^n) ,
\]
which we will denote by $\Box$.  Evidently each term except 1 in
$\Box$ has an initial factor of form $C(\la^k z^k)$, so
\[
\Box-1=C(\la z)\Box + C(\la^2 z^2)\Box +\cdots = 
C(\la z+\la z^2+\cdots)\Box ,
\]
and Eq. (\ref{ecid}) follows.
\end{proof}
Since $\exp:(\kA,\sh)\to (\kA,*)$ is an algebra isomorphism, we 
have $\exp f_{\sh}=f_*\exp$ for any $f\in tk[[t]]$.  
For such $f$ we also have $\Si f_{\star}=f_*\Si$ and $T f_*=f_{\star} T$.
In particular, for $z\in kA[[\la]]$, $\Si f_{\star}(\la z)=f_*(\la z)$ and
$Tf_*(\la z)=f_{\star}(-\la z)$.
For $z\in kA[[\la]]$ we also have
\begin{equation}
\label{expid}
\exp_*(\la z)=\exp(\exp_{\sh}(\la z))=\exp\left(\frac1{1-\la z}\right) ,
\end{equation}
where we have used the identity
\[
\exp_{\sh}(\la z)= 1 + \la z + \la^2z^2 + \la^3z^3 +\cdots=\frac1{1-\la z} ,
\]
which in turn follows from $z^{\sh n}=n!z^n$ for $z\in kA[[\la]]$.
We can now give a quick proof of the following result 
(cf. \cite[Prop. 4]{IKZ} and \cite[Prop. 3]{IKOO}).
\begin{cor}
\label{expthm}
For $z\in kA[[\la]]$, 
\[
\exp_*(\log_{\op}(1+\la z))=\frac1{1-\la z}
\quad\text{and}\quad
\exp_{\star}(-\log_{\op}(1+\la z))=\frac1{1+\la z} .
\]
\end{cor}
\begin{proof}
In view of Eq. (\ref{expid}), the first identity is equivalent to
\begin{equation}
\label{logid}
\frac1{1-\log_{\op}(1+\la z)}=\log\left(\frac1{1-\la z}\right) ,
\end{equation}
which is just Theorem \ref{Ihafid} applied to the 
formal power series $f=\log(1+t)$.
To get the second identity, apply $T$ to both sides of the first.
\end{proof}
\begin{remark}
By applying Theorem \ref{Ihafid} to the formal power
series $e^t-1$, we get
\[
\exp\left(\frac1{1-\la z}\right)=\frac1{1-(\exp_{\op}(\la z)-1)},
\]
or $\exp_*(\la z)=(2-\exp_{\op}(\la z))^{-1}$, as in \cite[Prop. 4]{IKZ}.
\end{remark}
Here are some further corollaries of Theorem \ref{Ihafid}.
\begin{cor}
\label{Hpow}
For $p\in k$ and $z\in kA[[\la]]$, 
\[
H_p\left(\frac1{1-\la z}\right)=\left(\frac1{1-\la z}\right)^{*p} .
\]
\end{cor}
\begin{proof}
Using Theorem \ref{hpthm} and Corollary \ref{expthm}, the left-hand side of 
the identity can be written
\[
H_p(\exp_*(\log_{\op}(1+\la z)))=\exp_*(H_p(\log_{\op}(1+\la z))) .
\]
Now $\log_{\op}(1+\la z)\in kA[[\la]]$, so the latter quantity is
\[
\exp_*(p\log_{\op}(1+\la z))=(\exp_*(\log_{\op}(1+\la z)))^{*p}=
\left(\frac1{1-\la z}\right)^{*p} .
\]
\end{proof}
\begin{cor}
\label{psifinv}
For any $p\in k$, $f=c_1t+c_2t^2+\dots\in tk[[t]]$ and $z\in kA[[\la]]$,
\[
\Psi_g\left(\frac1{1+\la z}\right)*
\left(\Psi_f\left(\frac1{1-\la z}\right)\right)^{*p}=1,
\]
where $(1+g(t))(1+f(-t))^p=1$, i.e., 
$g=((1+t)^{-p}-1)\circ f\circ(-t)$.
\end{cor}
\begin{proof}
By Theorem \ref{funct} $\Psi_g=H_{-p}\Psi_f T$, so 
the conclusion can be written as $H_{-p}(\xi)*\xi^{*p}=1$ for
\[
\xi=\Psi_f\left(\frac1{1-\la z}\right) .
\]
Now Theorem \ref{Ihafid} says that $\xi=\frac1{1-\la u}$ for
\[
u=c_1z+\la c_2 z^{\op 2} +\la^2 c_3 z^{\op 3}+\cdots \in kA[[\la]] ,
\]
so we can apply Corollary \ref{Hpow} to obtain the conclusion.
\end{proof}
\begin{remark}
In particular, taking $p=1$ and $f=\frac{t}{1-st}$ in the preceding 
result gives
\begin{equation}
\label{siinv}
\Si^s\left(\frac1{1-\la z}\right)*\Si^{1-s}\left(\frac1{1+\la z}\right) =1
\end{equation}
for any $s\in k$, generalizing Corollary 1 of \cite{IKOO}.
\end{remark}
\begin{cor}
\label{frprod}
For $y,z\in kA[[\la]]$, 
\[
\frac1{1-\la y}*\frac1{1-\la z}=\frac1{1-\la y-\la z-\la^2 y\op z}
\]
and
\[
\frac1{1+\la y}\star\frac1{1+\la z}=\frac1{(1+\la y)\op(1+\la z)} .
\]
\end{cor}
\begin{proof}
Using Corollary \ref{expthm}, the left-hand side of the first identity is
\begin{multline*}
\exp_*(\log_{\op}(1+\la y))*\exp_*(\log_{\op}(1+\la z))=
\exp_*(\log_{\op}(1+\la y)+\log_{\op}(1+\la z))\\
=\exp_*(\log_{\op}((1+\la y)\op (1+\la z)))=
\exp_*(\log_{\op}(1+\la y+\la z+\la^2 y\op z))\\
=\frac1{1-\la y-\la z-\la^2 y\op z},
\end{multline*}
so the identity follows.  To get the second identity, apply $T$ to 
both sides of the first.
\end{proof}
\begin{cor}
\label{repr}
For any $r\in k$ and $z\in kA[[\la]]$,
\[
\Si^r\left(\frac1{1-\la z}\right)*\frac1{1+r\la z}=\frac1{1-(1-r)\la z} .
\]
\end{cor}
\begin{proof}
By Theorem \ref{Ihafid}, 
\[
\Si^r\left(\frac1{1-\la z}\right)=\frac1{1-f_{\op}(\la z)}
\]
for $f=\frac{t}{1-rt}=t+rt^2+r^2t^3+\dots$.  Thus
\[
\Si^r\left(\frac1{1-\la z}\right)*\frac1{1+r\la z}=
\frac1{1-f_{\op}(\la z)+r\la z+\la rz\op f_{\op}(\la z)}
\]
by Corollary \ref{frprod}.  But evidently $\la rz\op f_{\op}(\la z)=
f_{\op}(\la z)-\la z$, so the denominator on the right-hand side is 
$1+r\la z-\la z$ and the conclusion follows.
\end{proof}
We also have the following result, which is proved in \cite[Prop. 4]{IKOO} 
by another method.
\begin{thm}
\label{dblfrac}
For $a,b\in A$,
\[
\Si\left(\frac1{1-\la ab}\right)=
\frac1{1-\la ab}*\Si\left(\frac1{1-\la a\op b}\right) .
\]
\end{thm}
\begin{proof}
Using Eq. (\ref{siinv}) with $s=1$, the conclusion can be written as
\[
\Si\left(\frac1{1-\la ab}\right)*\frac1{1+\la a\op b}
=\frac1{1-\la ab} .
\]
Now use Corollary \ref{siform} and apply $\log$ to both sides
to make this
\begin{equation}
\label{double}
T\log\left(\frac1{1-\la ab}\right)\sh\log\left(\frac1{1+\la a\op b}\right)
=\log\left(\frac1{1-\la ab}\right) .
\end{equation}
\par
Now
\[
\log\left(\frac1{1-\la ab}\right)=1+\sum_{i\ge 1}\la^i
\sum_{\substack{I=(i_1,\dots,i_n)\\ |I|=2i}}
\frac{(-1)^n}{i_1i_2\cdots i_n}I[(ab)^i] ,
\]
and applying $T$ simply eliminates the signs.  Further, 
\[
\log\left(\frac1{1+\la a\op b}\right)
=1+\sum_{i\ge 1} \la^i
\sum_{\substack{J=(j_1,\dots,j_k)\\ |J|=i}}\frac{(-1)^k}{j_1j_2\cdots j_k}J[(a\op b)^i] ,
\]
so to prove (\ref{double}) and hence the conclusion it suffices to show 
\begin{multline}
\label{mess}
\sum_{i=0}^m\left(\sum_{\substack{I=(i_1,\dots,i_n)\\ |I|=2i}}
\frac1{i_1\cdots i_n}I[(ab)^i]\right)
\sh\left(\sum_{\substack{J=(j_1,\dots,j_k)\\ |J|=m-i}}\frac{(-1)^k}{j_1\cdots j_k}
J[(a\op b)^{m-i}]\right)\\
=\sum_{\substack{I=(i_1,\dots,i_n)\\ |I|=2m}}\frac{(-1)^n}{i_1\cdots i_n}I[(ab)^m] .
\end{multline}
To prove the latter equation, we consider an arbitrary term of the
form
\begin{equation}
\label{word}
(i_1,i_2,\dots,i_n)[(ab)^m] ,\quad i_1+\dots+i_n=2m,
\end{equation}
and note that every even $i_h=2j$ produces a factor $(a\op b)^{\op j}$.
Write $(i_1,\dots, i_n)$ as $(t_1^{p_1},\dots,t_s^{p_s})$, where
the exponents mean repetition, and let $(t_{u_1},\dots,t_{u_f})=
(2j_{u_1},\dots,2j_{u_f})$ be the
subsequence of even $t_i$'s.  Then (\ref{word}) appears on the 
right-hand side of Eq. (\ref{mess}) with coefficient
\begin{equation}
\label{rhs}
\frac{(-1)^{p_1+\dots+p_s}}{t_1^{p_1}\cdots t_s^{p_s}}=
\frac{(-1)^{p_{u_1}+\dots+p_{u_f}}}{t_1^{p_1}\cdots t_s^{p_s}},
\end{equation}
and on the left-hand side of Eq. (\ref{mess}) with 
coefficient
\[
\sum_{0\le q_{u_h}\le p_{u_h}}
\frac1{t_1^{p_1'}\cdots t_s^{p_s'}}\frac{(-1)^{q_{u_1}+\dots+q_{u_f}}}
{j_{u_1}^{q_{u_1}}\cdots j_{u_f}^{q_{u_f}}}
\binom{p_{u_1}}{q_{u_1}}\cdots \binom{p_{u_f}}{q_{u_f}} ,
\]
where 
\[
p_i'=\begin{cases}
p_i, &\text{if $t_i$ is odd;}\\
p_i-q_i, &\text{if $t_i$ is even.}
\end{cases}
\]
Since $j_{u_h}^{q_{u_h}}=2^{-q_{u_h}}t_{u_h}^{q_{u_h}}$ for $h=1,2,\dots,f$, we
can write the latter coefficient as
\[
\frac1{t_1^{p_1}\cdots t_s^{p_s}}
\sum_{0\le q_{u_h}\le p_{u_h}}
(-2)^{q_{u_1}+\dots+q_{u_f}}
\binom{p_{u_1}}{q_{u_1}}\cdots \binom{p_{u_f}}{q_{u_f}} ,
\]
which by the binomial theorem agrees with (\ref{rhs}).
\end{proof}
\section{Applications}
To demonstrate the scope of applications of quasi-shuffle products,
in this section we will outline five types of objects that are 
homomorphic images of quasi-shuffle algebras:  multiple zeta values, 
multiple $t$-values, (finite) multiple harmonic sums, 
multiple $q$-zeta values, and values of multiple polylogarithms at 
roots of unity.
In each case we show how Corollary \ref{expthm} can be applied.
In several cases we also apply Theorem \ref{dblfrac}.
\subsection{Multiple zeta values}
\label{sMZV}
Consider the case $A=\{z_1,z_2,\dots\}$ and $z_i\op z_j=z_{i+j}$
(see \cite[Ex. 1]{H2}).
Let $\H^1=k\<A\>$, and let $\H^0\subset\H^1$ be the subspace generated
by monomials that don't start in $z_1$. 
Then there is a homomorphism $\zt: (\H^0,*)\to\R$ given as in the 
introduction:
\[
\zt(z_{i_1}z_{i_2}\cdots z_{i_k})=\sum_{n_1>n_2>\dots>n_k\ge 1}
\frac1{n_1^{i_1}n_2^{i_2}\cdots n_k^{i_k}} .
\]
(The restriction that $i_1\ne 1$ is necessary for convergence
of the series.)
Also, if we define $\zt^{\star}(z_{i_1}z_{i_2}\cdots z_{i_k})$ by Eq.
(\ref{mzvstar}), then $\zt^{\star}:(\H^0,\star)\to\R$ is a homomorphism.
\par
From Corollary \ref{expthm} we have
\begin{equation}
\label{expzk}
\exp_*\left(\sum_{i\ge 1}\frac{(-1)^{i-1}}{i}\la^i z_k^{\op i}\right)=\sum_{n=0}^\infty
z_k^n\la^n
\end{equation}
and applying $\zt$ to both sides gives
\[
\exp\left(\sum_{i\ge 1}\frac{(-1)^{i-1}}{i}\la^i\zt(ik)\right)=
\sum_{n=0}^\infty\zt(z_k^n)\la^n .
\]
That is, the MZV $\zt(k,\dots,k)$ (with $n$ repetitions of $k\ge 2$)
is the coefficient of $\la^n$ in
\begin{equation}
\label{ztpwr}
Z_k(\la)=\exp\left(\sum_{i\ge 1}\frac{(-1)^{i-1}}{i}\la^i\zt(ik)\right) .
\end{equation}
This is a well-known result:  it goes back at least to \cite{BBB}
(see Eq. (11)).  
To obtain the counterpart for zeta-star values, replace $\la$ 
with $-\la$ in the second part of Corollary \ref{expthm} 
and set $z=z_k$ to get
\begin{equation}
\label{expzks}
\exp_{\star}\left(\sum_{i\ge 1}\frac{\la^i z_k^{\op i}}{i}\right)=
\sum_{n=0}^\infty z_k^n\la^n .
\end{equation}
Now apply $\zt^{\star}$ to both sides:
\[
\exp\left(\sum_{i=1}^\infty\frac{\la^i\zt(ik)}{i}\right)=
\sum_{n=0}^\infty\zt^{\star}(z_k^n)\la^n .
\]
Henceforth any string enclosed by $\{\ \}_n$ is understood to be
repeated $n$ times, so the preceding equation implies that
$\zt(\{k\}_n)$ is the coefficient of $\la^n$ in 
\[
\exp\left(\sum_{i=1}^\infty\frac{\la^i\zt(ik)}{i}\right) =
\frac1{Z_k(-\la)} ,
\]
where $Z_k$ is given by Eq. (\ref{ztpwr}).
In particular, from \cite[Cor. 2.3]{H0} (and also \cite[Eq. (36)]{BBB})
we have 
\begin{equation}
\label{rep2}
\zt(\{2\}_n)=\frac{\pi^{2n}}{(2n+1)!},
\quad\text{hence}\quad
Z_2(\la)=\frac{\sinh(\pi\sqrt{\la})}{\pi\sqrt{\la}},
\end{equation}
and thus
\begin{equation}
\label{rep2s}
\sum_{n=0}^\infty \zt^{\star}(\{2\}_n)\la^n
=\frac{\pi\sqrt{\la}}{\sin(\pi\sqrt{\la})}.
\end{equation}
It follows that 
\[
\zt^{\star}(\{2\}_n)=\frac{(-1)^{n-1}2(2^{2n-1}-1)B_{2n}\pi^{2n}}{(2n)!} ;
\]
cf. \cite[p. 203]{IKOO}.  Similarly, we have from \cite[Eq. (37)]{BBB} that
\[
\zt(\{4\}_n)=\frac{2^{2n+1}\pi^{4n}}{(4n+2)!},
\quad\text{hence}\quad
Z_4(\la)=\frac{\cosh(\sqrt2\pi\sqrt[4]\la)-\cos(\sqrt2\pi\sqrt[4]\la)}
{2\pi^2\sqrt\la},
\]
and thus
\begin{multline*}
\sum_{n=0}^\infty \zt^{\star}(\{4\}_n)\la^n=\frac1{Z_4(-\la)}=
\frac{2\pi^2i\sqrt{\la}}{\cosh(\sqrt2\pi e^{i\pif}\sqrt[4]\la)-
\cos(\sqrt2\pi e^{i\pif}\sqrt[4]\la)}\\
=\frac{\pi^2\sqrt{\la}}{\sinh(\pi\sqrt[4]\la)\sin(\pi\sqrt[4]\la)} 
=1+\frac{\pi^4\la}{90}+\frac{13\pi^8\la^2}{113400}+\frac{4009\pi^{12}\la^3}
{3405402000}+\cdots .
\end{multline*}
\par
Applying $\zt$ to Theorem \ref{dblfrac} (with $a=z_i$, $b=z_j$) gives
\begin{equation}
\label{eqij}
\sum_{n=0}^\infty \zt^{\star}(\{i,j\}_n)\la^n=\sum_{p=0}^\infty \zt(\{i,j\}_p)\la^p
\sum_{q=0}^\infty \zt^{\star}(\{i+j\}_q)\la^q
\end{equation}
for any positive integers $i,j$ with $i\ge 2$.
In particular, taking $i=2$ and $j=1$ gives
\begin{equation}
\label{3gf}
\sum_{n=0}^\infty \zt^{\star}(\{2,1\}_n)\la^n=\sum_{p=0}^\infty \zt(\{2,1\}_p)\la^p
\sum_{q=0}^\infty \zt^{\star}(\{3\}_q)\la^q=\frac{Z_3(\la)}{Z_3(-\la)}
\end{equation}
since $\zt(\{2,1\}_p)=\zt(\{3\}_p)$ by duality of multiple zeta values.
Now
\begin{multline*}
\frac{Z_3(\la)}{Z_3(-\la)}=
\exp\left(\sum_{i\ge 1}\frac{(-1)^{i-1}}{i}\la^i\zt(3i)\right)
\exp\left(\sum_{i\ge 1}\frac{\la^i\zt(3i)}{i}\right)\\
=\exp\left(\sum_{\text{$i\ge 1$ odd}}\frac{2\la^i\zt(3i)}{i}\right)
=\prod_{\text{$i\ge 1$ odd}}\sum_{j=0}^\infty\frac{2^j\la^{ij}\zt(3i)^j}{i^jj!},
\end{multline*}
so it follows from Eq. (\ref{3gf}) that
\begin{equation}
\label{3odd}
\zt^{\star}(\{2,1\}_n)
=\sum_{i_1+3i_3+5i_5+\cdots=n}\frac{2^{i_1+i_3+i_5+\cdots}\zt(3)^{i_1}\zt(9)^{i_3}
\zt(15)^{i_5}\cdots}{1^{i_1}i_1!3^{i_3}i_3!5^{i_5}i_5!\cdots} .
\end{equation}
Similarly, using the Zagier-Broadhurst identity 
\cite[Theorem 1]{BBBL}
\[
\zt(\{3,1\}_n)=\frac{2\pi^{4n}}{(4n+2)!},
\]
we obtain from Eq. (\ref{eqij}) with $i=3$, $j=1$,
\begin{multline*}
\sum_{n=0}^\infty \zt^{\star}(\{3,1\}_n)\la^n
=\sum_{p=0}^\infty\zt(\{3,1\}_p)\la^p
\sum_{q=0}^\infty\zt^{\star}(\{4\}_q)\la^q
=\frac{Z_4\left(\frac{\la}{4}\right)}{Z_4(-\la)}\\
=\frac{\cosh(\pi\sqrt[4]\la)-\cos(\pi\sqrt[4]\la)}
{\sinh(\pi\sqrt[4]\la)\sin(\pi\sqrt[4]\la)}
=1+\frac{\pi^4\la}{72}+\frac{53\pi^8\la^2}{362880}+\frac{15107\pi^{12}\la^3}
{10059033600}+\cdots
\end{multline*}
\par
S. Yamamoto \cite{Y1} defines interpolated multiple zeta values
$\zt^r(i_1,\dots,i_k)$ as, in effect, $\zt(\Si^r(z_{i_1}\cdots z_{i_k}))$.
Note that $\zt^0=\zt$ and $\zt^1=\zt^{\star}$.
By applying $\zt$ to both sides of Corollary \ref{repr} with 
$z=z_k$, $k\ge 2$, we obtain
\begin{equation}
\label{intrep}
\sum_{n=0}^\infty \zt^r(\{k\}_n)\la^n=\frac{Z_k((1-r)\la)}{Z_k(-r\la)} .
\end{equation}
In particular, taking $r=\frac12$ in Eq. (\ref{intrep}) gives
\[
\sum_{n=0}^\infty \zt^{\frac12}(\{k\}_n)\la^n
=\frac{Z_k(\frac{\la}2)}{Z_k(-\frac{\la}2)}
=\exp\left(\sum_{\text{$i\ge 1$ odd}}\frac{\la^i\zt(ik)}{i2^{i-1}}\right)
=\prod_{\text{$i\ge 1$ odd}}\sum_{j=0}^\infty\frac{\la^{ij}\zt(ik)^j}{i^j2^{j(i-1)}j!} ,
\]
from which follows 
\begin{equation}
\label{khalf}
\zt^{\frac12}(\{k\}_n)=\sum_{i_1+3i_3+\cdots=n}
\frac{2^{i_1+i_3+i_5\cdots-n}\zt(k)^{i_1}\zt(3k)^{i_3}\zt(5k)^{i_5}\cdots}
{i_1!1^{i_1}i_3!3^{i_3}i_5!5^{i_5}\cdots} .
\end{equation}
Comparing the case $k=3$ of Eq. (\ref{khalf}) with Eq. (\ref{3odd}) above 
gives
\[
\zt^{\frac12}(\{3\}_n)=\frac1{2^n}\zt^{\star}(\{2,1\}_n),
\]
an instance of the two-one formula discussed in \cite{Y1} and recently 
proved in \cite{Zh3}.
\par
The sum theorem for multiple zeta values \cite{G} states that 
$\zt(S(k,l))=\zt(k)$ for $1\le l\le k-1$, where $S(k,l)$ is the sum of 
all monomials in $\H^0$ of degree $k$ and length $l$.  
Now it is a straightforward combinatorial exercise to show
that $DS(k,l)=(k-l)S(k,l-1)$, where $D$ is the canonical derivation of
\S4.  Then using Corollary \ref{dform},
\begin{multline*}
\zt^r(S(k,l))=\zt(\Si^rS(k,l))=\zt\left(\sum_{n=0}^{l-1}\frac{r^nD^n}{n!}S(k,l)
\right)=\\
\zt\left(\sum_{n=0}^{l-1}r^n\binom{k-l-1+n}{n}S(k,l-n)\right)
=\sum_{n=0}^{l-1}r^n\binom{k-l-1+n}{n}\zt(k),
\end{multline*}
giving a new proof of \cite[Theorem 1.1]{Y1}.
\par
In \cite{H4} it is proved that if 
\begin{equation}
\label{little}
e(2n,k)=\sum_{(i_1,\dots,i_k)\in\CC(n)}z_{2i_1}\cdots z_{2i_k}
\end{equation}
where the sum is over all compositions $(i_1,\dots,i_k)$ of $n$ having
$k$ parts, then the generating functions
\[
F(t,s)=1+\sum_{n\ge k\ge 1}\zt(e(2n,k))t^ns^k,\quad
F^{\star}(t,s)=1+\sum_{n\ge k\ge 1}\zt^{\star}(e(2n,k))t^ns^k
\]
have the closed forms
\[
F(t,s)=\frac{\sin(\pi\sqrt{(1-s)t})}{\sqrt{1-s}\sin(\pi\sqrt{t})},\quad
F^{\star}(t,s)=\frac{\sqrt{1+s}\sin(\pi\sqrt{t})}
{\sin(\pi\sqrt{(1+s)t})}.
\]
If we define
\[
F(t,s;r)=1+\sum_{n\ge k\ge 1}\zt^r(e(2n,k))t^ns^k,
\]
then this result can be generalized to 
\begin{equation}
\label{intev}
F(t,s;r)=\frac{\sin(\pi\sqrt{(1-s+rs)t})\sqrt{1+rs}}
{\sin(\pi\sqrt{(1+rs)t})\sqrt{1-s+rs}}=
\frac{F(t,(1-r)s)}{F(t,-rs)} .
\end{equation}
\par
To prove Eq. (\ref{intev}), we first note that $(k\<A\>,*)$ is isomorphic
to the algebra $\QS$ of quasi-symmetric functions.
Further, $e(2n,k)$ is the image under the degree-doubling map 
$\D:\QS\to\QS$ that sends $z_{i_1}\cdots z_{i_k}$
to $z_{2i_1}\cdots z_{2i_k}$ of the symmetric function
\[
N_{n,k}=\sum_{\substack{\text{partitions $\la$ of $n$}\\ \text{with $k$ parts}}}m_{\la} ,
\]
and, as shown in \cite{H4},
\[
1+\sum_{n\ge k\ge 1}N_{n,k}t^ns^k=E((s-1)t)H(t),
\]
where $E(t)$, $H(t)$ are respectively the generating functions of the
elementary and complete symmetric functions.  From Corollary \ref{repr} 
we have, for any rational $p$,
\[
\Si^p\left(\frac1{1-tz_1}\right)=\frac1{1-(1-p)tz_1}*\left(\frac1{1+ptz_1}
\right)^{-*},
\]
which translated into the language of symmetric functions is
\[
\Si^pE(t)=E((1-p)t)E(-pt)^{-1}=E((1-p)t)H(pt) .
\]
Then
\begin{multline*}
\Si^r\left(1+\sum_{n\ge k\ge 1}N_{n,k}t^ns^k\right)=
\Si^rE((s-1)t)H(t)=\Si^r\Si^{\frac1{s}}E(st)=\\
\Si^{r+\frac1{s}}E(st)=E((s-rs-1)t)H((1+rs)t) .
\end{multline*}
From Eqs. (\ref{rep2}) and (\ref{rep2s}) above we have
\[
\zt\D E(t)=\frac{\sinh(\pi\sqrt t)}{\pi\sqrt t}\quad\text{and}\quad
\zt\D H(t)=\frac{\pi\sqrt t}{\sin(\pi\sqrt t)}
\]
and so, since $\D$ commutes with $\Si^r$,
\begin{multline*}
F(t,s;r)=\zt\Si^r\D\left(1+\sum_{n\ge k\ge 1}N_{n,k}t^ns^k\right)=\\
\zt(\D(E((s-rs-1)t)H((1+rs)t)))
=\frac{\sinh(\pi\sqrt{(s-rs-1)t})\sqrt{(1+rs)t}}{\sqrt{(s-rs-1)t}\sin(\pi
\sqrt{(1+rs)t})},
\end{multline*}
from which Eq. (\ref{intev}) follows.
\subsection{Multiple $t$-values}
As in \cite{H5}, let $t(i_1,\dots,i_k)$ be the sum of those
terms in the series (\ref{mzv}) for $\zt(i_1,\dots,i_k)$ with 
odd denominators.
Then one has a homomorphism $t:(\H^0,*)\to\R$ defined by
$t(z_{i_1}\cdots z_{i_k})=t(i_1,\dots,i_k)$ for $i_1>1$.
We can also define multiple $t$-star values by
\[
t^{\star}(i_1,\dots,i_k)=\sum_{n_1\ge n_2\ge\cdots\ge n_k\ge 1,\ \text{$n_i$ odd}}
\frac1{n_1^{i_1}n_2^{i_2}\cdots n_k^{i_k}}
\]
for $i_1>1$, so that there is a homomorphism $t^{\star}:(\H^0,\star)\to\R$
given by $t^{\star}(z_{i_1}\cdots z_{i_k})=t^{\star}(i_1,\dots,i_k)$.
Applying $t$ to both sides of Eq. (\ref{expzk}) gives
\begin{equation}
\label{trep}
\exp\left(\sum_{i\ge 1}\frac{(-1)^{i-1}}{i}\la^it(ik)\right)=
\sum_{n=0}^\infty t(z_k^n)\la^n .
\end{equation}
Now 
\[
t(ik)=\sum_{\text{$n\ge 1$  odd}}\frac1{n^{ik}}=\left(1-\frac1{2^{ik}}\right)\zt(ik)
\]
so that the left-hand side of Eq. (\ref{trep}) is
\begin{multline*}
\exp\left(\sum_{i\ge 1}\frac{(-1)^{i-1}}{i}\left(1-\frac1{2^{ik}}\right)\la^i\zt(ik)
\right)=\\
\exp\left(\sum_{i\ge1}\frac{(-1)^{i-1}}{i}\la^i\zt(ik)\right)
\exp\left(-\sum_{i\ge1}\frac{(-1)^{i-1}}{i}\left(\frac{\la}{2^k}\right)^i
\zt(ik)\right)
\end{multline*}
and thus
\[
t(\{k\}_n)=\text{coefficient of $\la^n$ in}\
\frac{Z_k(\la)}{Z_k(\la/2^k)},
\]
where $Z(\la)$ is given by Eq. (\ref{ztpwr}); cf. \cite[Theorem 8]{H5}.  
Similarly
\[
t^{\star}(\{k\}_n)=\text{coefficient of $\la^n$ in}\
\frac{Z_k(-\la/2^k)}{Z_k(-\la)} .
\]
In particular, from the preceding subsection we have
\begin{equation}
\label{t2}
\sum_{n=0}^\infty t(\{2\}_n)\la^n=\cosh\left(\pit\sqrt{\la}\right),\quad
\sum_{n=0}^\infty t^{\star}(\{2\}_n)\la^n=\sec\left(\pit\sqrt{\la}\right),
\end{equation}
and
\begin{align*}
\sum_{n=0}^\infty t(\{4\}_n)\la^n&=\frac12\left[\cosh\left(\frac{\pi\sqrt[4]\la}
{\sqrt2}\right)+\cos\left(\frac{\pi\sqrt[4]\la}{\sqrt2}\right)\right],\\
\sum_{n=0}^\infty t^{\star}(\{4\}_n)\la^n&=\sec\left(\pit\sqrt[4]\la\right)
\sech\left(\pit\sqrt[4]\la\right) .
\end{align*}
\par
As with multiple zeta values, we can define interpolated multiple
$t$-values
\[
t^r(i_1,\dots,i_k)=t(\Si^r(z_{i_1}\cdots z_{i_k})) .
\]
Reasoning as before, we obtain from Corollary \ref{repr} an
analogue of Eq. (\ref{intrep}):
\[
\sum_{n=0}^\infty t^r(\{k\}_n)\la^n=\frac{Z_k((1-r)\la)Z_k(-r\la/2^k)}
{Z_k((1-r)\la/2^k)Z_k(-r\la)} .
\]
In the case $r=\frac12$ this gives a formula like Eq. (\ref{khalf}):
\[
t^{\frac12}(\{k\}_n)=\sum_{i_1+3i_3+\cdots=n}
\frac{2^{i_1+i_3+i_5\cdots-n}t(k)^{i_1}t(3k)^{i_3}t(5k)^{i_5}\cdots}
{i_1!1^{i_1}i_3!3^{i_3}i_5!5^{i_5}\cdots} .
\]
Also, we can generalize the result of \cite{Zh2} that
\[
1+\sum_{n\ge k\ge 1}t(e(2n,k))\tau^n\si^k
=\frac{\cos\left(\pit\sqrt{(1-\si)\tau}\right)}
{\cos\left(\pit\sqrt \tau\right)}
\]
(where $e(2n,k)$ is given by Eq. (\ref{little}) above) to
\[
1+\sum_{n\ge k\ge 1}t^r(e(2n,k))\tau^n\si^k
=\frac{\cos\left(\pit\sqrt{(1-\si+r\si)\tau}\right)}
{\cos\left(\pit\sqrt{(1+r\si)\tau}\right)}
\]
using the reasoning of the last subsection and Eqs. (\ref{t2}).
\subsection{Multiple harmonic sums}
If one defines, for fixed $n$, the finite sums
\[
A_{(k_1,\dots,k_l)}(n)=\sum_{n\ge m_1>m_2>\dots>m_l\ge 1}
\frac1{m_1^{k_1}m_2^{k_2}\cdots m_l^{k_l}} 
\]
and 
\[
S_{(k_1,\dots,k_l)}(n)=\sum_{n\ge m_1\ge m_2\ge\dots\ge m_l\ge 1}
\frac1{m_1^{k_1}m_2^{k_2}\cdots m_l^{k_l}} ,
\]
then there are homomorphisms $\zt_{\le n}:(\H^1,*)\to\R$ and
$\zt_{\le n}^{\star}:(\H^1,\star)\to\R$ given by
\[
\zt_{\le n}(z_{k_1}\cdots z_{k_l})=A_{(k_1,\dots,k_l)}(n)
\]
and 
\[
\zt_{\le n}^{\star}(z_{k_1}\cdots z_{k_l})=S_{(k_1,\dots,k_l)}(n) .
\]
\par
Applying these homomorphisms to the Eqs. (\ref{expzk}) and (\ref{expzks}) 
above, we obtain
\[
\zt_{\le n}(\{k\}_{r})=\text{coefficient of $\la^r$ in}\
\exp\left(\sum_{i\ge 1}\frac{(-1)^{i-1}}{i}\la^i A_{ik}(n)\right) .
\]
and
\begin{equation}
\label{finsrep}
\zt_{\le n}^{\star}(\{k\}_{r})=\text{coefficient of $\la^r$ in}\
\exp\left(\sum_{i\ge 1}\frac{\la^i A_{ik}(n)}{i}\right) .
\end{equation}
Note that $k$ can be 1 in these formulas since the sums
involved are finite.
In particular, it is well-known that
\[
\zt_{\le n}(\{1\}_r)=\frac1{n!}\stone{n+1}{r+1} ,
\]
where $\stone{n}{k}$ is the number of permutations of $\{1,2,\dots,n\}$
with $k$ disjoint cycles (unsigned Stirling number of the first kind).
Eq. (\ref{finsrep}) can be compared to the explicit formula given 
by \cite[Eq. (21)]{H3}.
\subsection{Multiple $q$-zeta values}
As in the preceding examples let $A=\{z_1,z_2,\dots\}$, but now
define the product $\op$ by
\begin{equation}
\label{qop}
z_i\op z_j=z_{i+j}+ (1-q)z_{i+j-1} .
\end{equation}
Here we take as our ground field $k=\Q(1-q)$.
Let $\H_q^1=k\<A\>$, $\H_q^0$ the subspace generated by words
not starting with $z_1$.
Then we have homomorphisms $\zt_q:(\H_q^0,*)\to\Q[[q]]$ and
$\zt_q^{\star}:(\H_q^0,\star)\to\Q[[q]]$ given by
\[
\zt_q(z_{k_1}z_{k_2}\cdots z_{k_l})=\sum_{m_1>m_2>\dots>m_l\ge 1}
\frac{q^{m_1(k_1-1)+m_2(k_2-1)+\dots+m_l(k_l-1)}}
{[m_1]^{k_1}[m_2]^{k_2}\cdots [m_l]^{k_l}} ,
\]
and
\[
\zt_q^{\star}(z_{k_1}z_{k_2}\cdots z_{k_l})=\sum_{m_1\ge m_2\ge \dots>m_l\ge 1}
\frac{q^{m_1(k_1-1)+m_2(k_2-1)+\dots+m_l(k_l-1)}}
{[m_1]^{k_1}[m_2]^{k_2}\cdots [m_l]^{k_l}} ,
\]
where $[m]=(1-q^m)/(1-q)$.
\par
Formulas like those obtained in the last three examples are complicated
by presence of the extra term in Eq. (\ref{qop}).  Iteration of
(\ref{qop}) gives
\[
z_k^{\op i}=\sum_{j=0}^{i-1}\binom{i-1}{j}(1-q)^jz_{ik-j} .
\]
Then, as in \cite[Ex. 4]{IKOO}, we can apply $\zt_q$ and $\zt_q^{\star}$
to the equations 
\[
\exp_*\left(\sum_{i\ge 1}\frac{(-1)^{i-1}}{i}\la^iz_k^{\op i}\right)=\sum_{i=0}^\infty
z_k^i\la^i
\quad\text{and}\quad
\exp_{\star}\left(\sum_{i\ge 1}\frac{\la^i}{i}z_k^{\op i}\right)=\sum_{i=0}^\infty
z_k^i\la^i
\]
to get, for $k\ge 2$,
\begin{multline*}
\zt_q(\{k\}_r)=\\
\text{coefficient of $\la^r$ in}\
\exp\left[\sum_{i\ge 1}\frac{(-1)^{i-1}\la^i}{i}\left(\sum_{j=0}^{i-1}
\binom{i-1}{j}(1-q)^j\zt_q(ik-j)\right)\right]
\end{multline*}
and
\begin{multline*}
\zt_q^{\star}(\{k\}_r)=\\
\text{coefficient of $\la^r$ in}\
\exp\left[\sum_{i\ge 1}\frac{\la^i}{i}\left(\sum_{j=0}^{i-1}
\binom{i-1}{j}(1-q)^j\zt_q(ik-j)\right)\right] .
\end{multline*}
\subsection{Multiple polylogarithms at roots of unity}
Fix $r\ge 2$, and let $\om=e^{\frac{2\pi i}{r}}$.  Then for
an integer composition $I=(i_1,\dots,i_k)$, the values of the 
multiple polylogarithm $\L_I$ at $r$th roots of unity are given by
\[
\L_I(\om^{j_1},\dots,\om^{j_k})=
\sum_{n_1>\cdots >n_k\ge 1}\frac{\om^{n_1j_1}\cdots \om^{n_kj_k}}{n_1^{i_1}\cdots n_k^{i_k}},
\]
and the series converges provided $\om^{j_1}i_1\ne 1$.
We can define the multiple ``star-polylogarithms'' by
\[
\L_I^{\star}(\om^{j_1},\dots,\om^{j_k})=\sum_{n_1\ge\cdots\ge n_k\ge 1}
\frac{\om^{n_1j_1}\cdots \om^{n_kj_k}}{n_1^{i_1}\cdots n_k^{i_k}} .
\]
Here we let $A=\{z_{i,j} : i \ge 1, 0\le j\le r-1\}$ and
$z_{i,j}\op z_{p,q}=z_{i+p,j+q}$, where the second subscript is
understood mod $r$.  The algebra $(\kA,*)$ is called the
Euler algebra in \cite{H2} (see Ex. 2).
Let $\E_r=\kA$, $\E_r^0$ the subalgebra of $\kA$ generated by words
not starting in $z_{1,0}$. 
Then there are homomorphisms $Z:(\E_r^0,*)\to\C$ and
$Z^{\star}:(\E_r^0,\star)\to\C$ given by
\begin{align*}
Z(z_{i_1,j_1}\cdots z_{i_k,j_k})&=\L_{(i_1,\dots,i_k)}(\om^{j_1},\dots,\om^{j_k})\\
Z^{\star}(z_{i_1,j_1}\cdots z_{i_k,j_k})&=\L_{(i_1,\dots,i_k)}^{\star}
(\om^{j_1},\dots,\om^{j_k}) .
\end{align*}
\par
From Corollary \ref{expthm} we have
\[
\exp_*\left(\sum_{i\ge 1} \frac{(-1)^{i-1}\la^i}{i} z_{s,t}^{\op i}\right)=
\sum_{i=0}^\infty z_{s,t}^i \la^i .
\]
Now $z_{s,t}^{\op i}=z_{si,ti}$, and if $t$ is relatively prime to $r$ the 
preceding equation is
\[
\exp_*\left(\sum_{j=0}^{r-1}\sum_{\substack{i\ge 1\\ ti\equiv j\mod r}} 
\frac{(-1)^{i-1}\la^i}{i} z_{is,j}\right)
=\sum_{i=0}^\infty z_{s,t}^i \la^i ,
\]
with fewer terms on the left-hand side if $t$ has factors in common
with $r$.  Applying $Z$ to both sides, we have
\begin{multline}
\label{polyl}
\L_{\underbrace{(s,\dots,s)}_k}(\underbrace{\om^t,\dots,\om^t}_k)=\\
\text{coefficient of $\la^k$ in}\
\exp\left(\sum_{j=0}^{r-1}\sum_{\substack{i\ge 1\\ ti\equiv j\mod r}} 
\frac{(-1)^{i-1}\la^i}{i} \L_{is}(\om^j)\right) .
\end{multline}
The counterpart for multiple star-polylogarithms is
\begin{multline}
\label{polyls}
\L_{\underbrace{(s,\dots,s)}_k}^{\star}(\underbrace{\om^t,\dots,\om^t}_k)=\\
\text{coefficient of $\la^k$ in}\
\exp\left(\sum_{j=0}^{r-1}\sum_{\substack{i\ge 1\\ ti\equiv j\mod r}} 
\frac{\la^i}{i} \L_{is}(\om^j)\right) .
\end{multline}
\par
The simplest case is $r=2$.  Here $\om=-1$ and images under $Z$ and 
$Z^{\star}$ are the alternating or ``colored'' MZVs (which lie in $\R$).
In this case we can streamline the notation and write, e.g., 
$\zt(\bar2,1)$ instead of $\L_{(2,1)}(-1,1)$.
Note that $\zt(\bar 1)=-\log 2$ and $\zt(\bar s)=(2^{1-s}-1)\zt(s)$
for $s\ge 2$.
For $r=2$, $t=1$ and $s\ge 2$, Eq. (\ref{polyl}) above simplifies to
\begin{align*}
\sum_{k=0}^\infty\zt(\{\bar s\}_k)\la^k&=
\exp\left(-\sum_{\text{$i$ even}}\frac{\la^i}{i}\zt(is)+
\sum_{\text{$i$ odd}}\frac{\la^i}{i}\zt(\overline{is})\right)\\
&=\exp\left(-\sum_{i=1}^\infty\frac{\la^i}{i}\zt(is)+2\sum_{\text{$i$ odd}}
\left(\frac{\la}{2^s}\right)^i\frac{\zt(is)}{i}\right)\\
&=Z_s(-\la)\exp\left(\sum_{\text{$i$ odd}}\left(\frac{\la}{2^s}\right)^i
\frac{\zt(is)}{i}\right)^2 ,
\end{align*}
where $Z_s(\la)$ is defined by Eq. (\ref{ztpwr}) above.
Now
\[
\sum_{\text{$i$ odd}}\left(\frac{\la}{2^s}\right)^i\frac{\zt(is)}{i}=
\sum_{i=1}^\infty\left(\frac{\la}{2^s}\right)^i(-1)^{i-1}\frac{\zt(is)}{i}
+\sum_{i=1}^\infty\left(\frac{\la}{2^s}\right)^{2i}\frac{\zt(2is)}{2i}
\]
so that
\[
\exp\left(
\sum_{\text{$i$ odd}}\left(\frac{\la}{2^s}\right)^i\frac{\zt(is)}{i}
\right)=\frac{Z_s\left(\frac{\la}{2^s}\right)}
{\sqrt{Z_{2s}\left(-\frac{\la^2}{4^s}\right)}}
\]
and thus
\[
\zt(\{\bar s\}_k)=\text{coefficient of $\la^k$ in}\
\frac{Z_s(-\la)Z_s\left(\frac{\la}{2^s}\right)^2}{Z_{2s}\left(-\frac{\la^2}{4^s}
\right)} ;
\]
cf. \cite[Eq. (12)]{BBB}.  In particular
\[
\sum_{k=0}^\infty\zt(\{\bar 2\}_k)\la^k=
\frac{Z_2(-\la)Z_2\left(\frac{\la}4\right)^2}
{Z_4\left(-\frac{\la^2}{16}\right)}=
\frac{2\cos(\pit\sqrt\la)\sinh(\pit\sqrt\la)}{\pi\sqrt\la} .
\]
Similarly, for $r=2$, $t=1$ and $s\ge 2$, Eq. (\ref{polyls}) above gives
\begin{multline*}
\sum_{k=0}^\infty\zt^{\star}(\{\bar s\}_k)\la^k=
\exp\left(\sum_{\text{$i$ even}}\frac{\la^i}{i}\zt(is)+
\sum_{\text{$i$ odd}}\frac{\la^i}{i}\zt(\overline{is})\right)\\
=\exp\left(\sum_{i=1}^\infty\frac{(-\la)^i}{i}\zt(is)+2\sum_{\text{$i$ odd}}
\left(\frac{\la}{2^s}\right)^i\frac{\zt(is)}{i}\right)
=\frac{Z_s\left(\frac{\la}{2^s}\right)^2}{Z_s(\la)Z_{2s}\left(-\frac{\la^2}{4^s}
\right)},
\end{multline*}
which for $s=2$ is
\[
\sum_{k=0}^\infty\zt^{\star}(\{\bar 2\}_k)\la^k=
\frac{Z_2\left(\frac{\la}{4}\right)^2}{Z_2(\la)Z_4\left(-\frac{\la^2}{16}\right)}
=\frac{\pi\sqrt\la}{2\cosh(\pit\sqrt\la)\sin(\pit\sqrt\la)} .
\]
In the case $s=1$ we have
\begin{multline}
\label{s1}
\sum_{k=0}^\infty\zt(\{\bar 1\}_k)\la^k=
\exp\left(-\sum_{\text{$i$ even}}\frac{\la^i}{i}\zt(i)+
\sum_{\text{$i$ odd}}\frac{\la^i}{i}\zt(\bar i)\right)=\\
\exp\left(-\sum_{i=2}^\infty\frac{\la^i}{i}\zt(i)-\la\log2+2\sum_{\text{$i\ge3$ odd}}
\left(\frac{\la}{2}\right)^i\frac{\zt(i)}{i}\right)
=\frac{Z_1(-\la)}{2^{\la}}\frac{Z_1\left(\frac{\la}2\right)^2}
{Z_2\left(-\frac{\la^2}4\right)} ,
\end{multline}
with $Z_1(\la)$ interpreted as
\[
\exp\left(\sum_{i=2}^\infty\frac{(-1)^{i-1}}{i}\la^i\zt(i)\right)=
\frac1{e^{\ga\la}\Ga(1+\la)},
\]
where $\ga$ is Euler's constant and $\Ga$ is the gamma function.
Using the duplication and reflection formulas for the gamma function,
Eq. (\ref{s1}) gives
\[
\sum_{k=0}^\infty\zt(\{\bar 1\}_k)\la^k=\frac{\sqrt\pi}{\Ga\left(\frac{1-\la}2
\right)\Ga\left(1+\frac{\la}2\right)};
\]
cf. \cite[Eq. (13)]{BBB}.
Similarly
\[
\sum_{k=0}^\infty\zt^{\star}(\{\bar 1\}_k)\la^k=
\frac1{\sqrt\pi}\Ga\left(\frac{1+\la}2\right)\Ga\left(1-\frac{\la}2\right) .
\]
\par
The remarkable identity 
\begin{equation}
\label{zhao}
\zt(\{\bar 2,1\}_n)=\frac1{8^n}\zt(\{3\}_n),
\end{equation}
was conjectured in \cite{BBB} but only proved by J. Zhao \cite{Zh1} 
13 years later.
As in Subsection \ref{sMZV} above, we can apply Theorem \ref{dblfrac}
(this time with $a=z_{2,1}$ and $b=z_{1,0}$) to get
\[
\sum_{n=0}^\infty\zt^{\star}(\{\bar2,1\}_n)\la^n=
\sum_{p=0}^\infty\zt(\{\bar2,1\}_p)\la^p
\sum_{q=0}^\infty\zt^{\star}(\{\bar3\}_q)\la^q,
\]
which by Eq. (\ref{zhao}) and the results above is
\begin{multline*}
\sum_{n=0}^\infty\zt^{\star}(\{\bar2,1\}_n)\la^n
=\frac{Z_3\left(\frac{\la}{8}\right)^3}
{Z_3(\la)Z_6\left(-\frac{\la^2}{64}\right)} =\\
\exp\left(\sum_{\text{$i\ge 1$ odd}}\frac{\zt(3i)\la^i}{i8^i}(3-8^i)\right)
\exp\left(\sum_{\text{$i\ge 2$ even}}\frac{\zt(3i)\la^i}{i8^i}(8^i-1)\right) .
\end{multline*}
\section*{Acknowledgments}
The authors would like to thank the directors of the Max-Planck-Institut
f\"ur Mathematik (MPIM) in Bonn for their hospitality in early 2012, when 
they began work on this project and wrote a preliminary report \cite{HI}.
In 2013 and 2014, the first author received partial support from the
Naval Academy Research Council.  The second author was supported by
JSPS KAKENHI grant number 26400013.  The authors thank the referee for
his or her careful reading and thoughtful review of the manuscript.


\begin{thebibliography}{9}
\bibitem{A}
M. Aguiar, Infinitesimal Hopf algebras, in 
\emph{New Trends in Hopf Algebra Theory} (La Falda, 1999), 
N. Andruskiewitsch {\it et. al.} (eds.), Contemporary Math. 267, 
Amer. Math. Soc., Providence, RI, 2000, pp. 1-29.
\bibitem{BBB}
J. M. Borwein, D. M. Bradley, and D. J. Broadhurst, Evaluation of
$k$-fold Euler/Zagier sums:  a compendium for arbitrary $k$,
\emph{Electron. J. Combin.} {\bf 4(2)} (1997), paper R5.
\bibitem{BBBL}
J. M. Borwein, D. M. Bradley, D. J. Broadhurst, and P. Lison\v ek,
Combinatorial aspects of multiple zeta values, \emph{Electron. J. Combin.} 
{\bf 5} (1998), paper R38.
\bibitem{B}
D. M. Bradley, Multiple $q$-zeta values, \emph{J. Algebra} {\bf 283}
(2005), 752-798.
\bibitem{G}
A. Granville, A decomposition of Riemann's zeta function, in 
\emph{Analytic Number Theory}, Y. Motohashi (ed.), London Math. Soc.
Lect. Notes 247, Cambridge University Press, Cambridge, 1997, pp. 95-101.
\bibitem{H0}
M. E. Hoffman, Multiple harmonic series, \emph{Pacific J. Math.} {\bf 152}
(1992), 175-190.
\bibitem{H2}
M. E. Hoffman, Quasi-shuffle products, \emph{J. Algebraic Combin.}
{\bf 11} (2000), 49-68.
\bibitem{H3}
M. E. Hoffman, The Hopf algebra structure of multiple harmonic sums,
\emph{Nuclear Phys. B (Proc. Suppl.)} {\bf 135} (2004), 215-219.
\bibitem{H4}
M. E. Hoffman, On multiple zeta values of even arguments, \emph{Int. J. 
Number Theory} {\bf 13} (2017), 705-716.
\bibitem{H5}
M. E. Hoffman, An odd variant of multiple zeta values, preprint
{\tt arXiv:1612.06563}.
\bibitem{HI}
M. E. Hoffman and K. Ihara, Quasi-shuffle products revisited,
MPIM preprint 2012-16.
\bibitem{IKOO}
K. Ihara, J. Kajikawa, Y. Ohno, and J. Okuda, Multiple zeta values vs.
multiple zeta-star values, \emph{J. Algebra} {\bf 332} (2011), 
187-208.
\bibitem{IKZ}
K. Ihara, M. Kaneko, and D. Zagier, Derivation and double shuffle
relations for multiple zeta values, \emph{Compositio Math.} {\bf 142}
(2006), 307-338.
\bibitem{J}
N. Jacobson, \emph{Lie Algebras}, Dover, New York, 1979.
\bibitem{M}
D. Manchon, Hopf algebras in renormalization, in \emph{Handbook of Algebra}, 
vol. 5, M. Hazewinkel (ed.), Elsevier, Amsterdam, 2008, pp. 365-427.
\bibitem{Y1}
S. Yamamoto, Interpolation of multiple zeta values and zeta-star values,
\emph{J. Algebra} {\bf 385} (2013), 102-114.
\bibitem{Y2}
S. Yamamoto, Hoffman-Ihara operators and the formal power series operad
[Japanese], Proceedings of the 6th Workshop on Multiple Zeta (Feb. 22-24,
2013), Kyushu Univ., Japan.
\bibitem{Zh1}
J. Zhao, On a conjecture of Borwein, Bradley and Broadhurst, 
\emph{J. reine angew. Math.} {\bf 639} (2010), 223-233.
\bibitem{Zh2}
J. Zhao, Sum formula of multiple Hurwitz-zeta values, \emph{Forum Math.}
{\bf 27} (2015), 929-936.
\bibitem{Zh3}
J. Zhao, Identity families of multiple harmonic sums and multiple
zeta-star values, \emph{J. Math. Soc. Japan} {\bf 68} (2016), 1669-1694.
\end{thebibliography}
\end{document}